\documentclass[a4paper,12pt,reqno]{amsart}
\usepackage{amsmath,amsthm,amssymb}
\usepackage{graphicx, color}
\usepackage[numbers,sort&compress]{natbib}
\nonstopmode \numberwithin{equation}{section}

\newtheorem{theorem}{Theorem}[section]

 \newtheorem{corollary}{Corollary}[section]

\allowdisplaybreaks

\allowdisplaybreaks

\begin{document}

\title[An extension of the Whittaker function and its properties ]{An extension of the Whittaker function}

\author[G. Rahman,   K.S. Nisar, J. Choi ]{ Gauhar Rahman,  Kottakkaran Sooppy  Nisar, Junesang Choi*}

\address{Gauhar Rahman:    Department of Mathematics, International Islamic  University, Islamabad, Pakistan}
\email{gauhar55uom@gmail.com}

\address{Kottakkaran Sooppy  Nisar:    Department of Mathematics, College of Arts and Science-Wadi Aldawaser, 11991,
Prince Sattam bin Abdulaziz University,Alkharj, Kingdom of Saudi Arabia}
\email{n.sooppy@psau.edu.sa; ksnisar1@gmail.com}

\address{Junesang Choi: Department of Mathematics,
Dongguk University, Gyeongju
38066, Republic of Korea}
\email{junesang@mail.dongguk.ac.kr}



\thanks{*Corresponding author}

\begin{abstract}
The Whittaker function and its diverse extensions have been
actively investigated. Here we introduce an extension of
the Whittaker function by using the known extended confluent hypergeometric function
$\Phi_{p,v}$ and investigate some of its formulas such as
integral representations, a transformation formula, Mellin transform, and a differential formula.
Some special cases of our results are also considered.

 \end{abstract}

\maketitle

\section{Introduction and preliminaries}\label{sec1}

We begin by recalling  the classical beta function (see, e.g., \cite[p. 8]{Sr-Ch-12})
 \begin{equation}\label{beta}
  B(\alpha,\, \beta) = \left\{ \aligned & \int_0^1 \, t^{\alpha -1} (1-t)^{\beta -1} \, dt
    \quad (\Re(\alpha)>0; \,\, \Re(\beta)>0) \\
    &\frac{\Gamma (\alpha) \, \Gamma (\beta)}{\Gamma (\alpha+ \beta)}
   \hskip 23mm  \left(\alpha,\, \beta
      \in \mathbb{C}\setminus {\mathbb Z}_0^- \right). \endaligned \right.
\end{equation}
Here and in the following, let $\mathbb{C}$,  $\mathbb{R}$, $\mathbb{R}^+$, $\mathbb{N}$,     and $\mathbb{Z}_0^-$ be the sets of
complex numbers, real numbers, positive real numbers, positive integers, and non-positive integers, respectively,
 and let  $\mathbb{R}_0^+:=\mathbb{R}^+ \cup \{0\} $ and $\mathbb{N}_0:= \mathbb{N}\cup \{0\}$.

The Gauss hypergeometric function ${}_2F_1$ and the confluent hypergeometric function ${}_1\Phi_1$ are  defined by (see, e.g., \cite{Rainville1960}; see also \cite[Section 1.5]{Sr-Ch-12})
\begin{equation}\label{Hyper}
{}_2F_1(\sigma_1,\sigma_2;\sigma_3;z)=\sum\limits_{n=0}^{\infty}\frac{(\sigma_1)_n(\sigma_2)_n}{(\sigma_3)_n}\frac{z^n}{n!}
\end{equation}
 \begin{equation*}
  \Big(\sigma_1,\,\sigma_2\in\mathbb{C}, \, \sigma_3 \in \mathbb{C}\setminus {\mathbb Z}_0^-;\, |z|<1  \Big)
 \end{equation*}
 and
 \begin{equation}\label{Chyper}
{}_1F_1(\sigma_2;\sigma_3;z)=\Phi(\sigma_2;\sigma_3;z)=\sum\limits_{n=0}^{\infty}\frac{(\sigma_2)_n}{(\sigma_3)_n}\frac{z^n}{n!}
\quad \Big(\sigma_2\in\mathbb{C}, \, \sigma_3 \in \mathbb{C}\setminus {\mathbb Z}_0^-;\, z \in \mathbb{C}  \Big),
\end{equation}
where $(\lambda)_n$ denotes the Pochhammer symbol (see, e.g., \cite[Section 1.1]{Sr-Ch-12}).
The well known  integral representations of the hypergeometric function  and the confluent hypergeometric functions are
recalled (see, e.g., \cite[Section 1.5]{Sr-Ch-12})
\begin{equation}\label{Ihyper}
{}_2F_1(\sigma_1,\sigma_2;\sigma_3;z)=\frac{\Gamma(\sigma_3)}{\Gamma(\sigma_2)\Gamma(\sigma_3-\sigma_2)}
\int_0^1t^{\sigma_2-1}(1-t)^{\sigma_3-\sigma_2-1}(1-zt)^{-\sigma_1}\,dt
\end{equation}
\begin{equation*}
  \big(\Re(\sigma_3)>\Re(\sigma_2)>0,\, |\arg(1-z)|<\pi \big)
\end{equation*}
 and
\begin{equation}\label{Ichyper}
\Phi(\sigma_2;\sigma_3;z)=\frac{\Gamma(\sigma_3)}{\Gamma(\sigma_2)\Gamma(\sigma_3-\sigma_2)}\int_0^1\,t^{\sigma_2-1}
(1-t)^{\sigma_3-\sigma_2-1}\, \mathtt{e}^{zt}\, dt
\end{equation}
\begin{equation*}
  \big(\Re(\sigma_3)>\Re(\sigma_2)>0 \big)
\end{equation*}

In last several decades, various extensions of some well-known special functions have been investigated.
For example, Chaudhary et al. \cite{Chaudhry1997} introduced the following extended beta function
\begin{eqnarray}\label{Ebeta}
B(\sigma_1,\sigma_2;p)=B_p(\sigma_1,\sigma_2)=\int_{0}^{1}\, t^{\sigma_1-1}(1-t)^{\sigma_2-1}\,
 \mathtt{e}^{-\frac{p}{t(1-t)}}\,dt
\end{eqnarray}
\begin{equation*}
  \big(\min\left\{\Re(p),\, \Re(\sigma_1),\,\Re(\sigma_2) \right\}>0  \big).
\end{equation*}
Obviously $B(\sigma_1,\sigma_2;0)= B(\sigma_1,\sigma_2)$.
 Also,  Chaudhry et al. \cite{Chaudhrya2004} introduced the  extended  hypergeometric function $F_p$
and the confluent hypergeometric function $\Phi_p$
\begin{eqnarray}\label{Ehyper}
F_p(\sigma_1,\sigma_2;\sigma_3;z)=\sum\limits_{n=0}^{\infty}\frac{B_p(\sigma_2+n, \sigma_3-\sigma_2)}{B(\sigma_2, \sigma_3-\sigma_2)}(\sigma_1)_n\frac{z^n}{n!}
\end{eqnarray}
\begin{equation*}
  \big(p\geq 0,\,|z|<1,\,  \Re(\sigma_3)>\Re(\sigma_2)>0 \big)
\end{equation*}
and
\begin{eqnarray}\label{ECH}
\Phi_p(\sigma_2;\sigma_3;z)=\sum\limits_{n=0}^{\infty}\frac{B_p(\sigma_2+n, \sigma_3-\sigma_2)}{B(\sigma_2, \sigma_3-\sigma_2)}\frac{z^n}{n!}
\end{eqnarray}
\begin{equation*}
  \big(p\geq 0,  \Re(\sigma_3)>\Re(\sigma_2)>0 \big).
\end{equation*}
They  \cite{Chaudhrya2004} presented the following integral representations
\begin{equation}\label{IEhyper}
\aligned
&F_p(\sigma_1,\sigma_2;\sigma_3;z)\\
&\hskip 3mm =\frac{1}{B(\sigma_2, \sigma_3-\sigma_2)}\int_0^1t^{\sigma_2-1}(1-t)^{\sigma_3-\sigma_2-1}(1-zt)^{-\sigma_1}\exp\Big(\frac{-p}{t(1-t)}\Big)\,dt
\endaligned
\end{equation}
$$\Big(p \in \mathbb{R}_0^+,\, \Re(\sigma_3)>\Re(\sigma_2)>0,\, |\arg(1-z)|<\pi\Big)$$
and
\begin{eqnarray}\label{IEChyper}
\Phi_p(\sigma_2;\sigma_3;z)=\frac{1}{B(\sigma_2, \sigma_3-\sigma_2)}\int_0^1t^{\sigma_2-1}(1-t)^{\sigma_3-\sigma_2-1}\exp\Big(zt-\frac{p}{t(1-t)}\Big)dt,
\end{eqnarray}
$$\Big(p \in \mathbb{R}_0^+,\, \Re(\sigma_3)>\Re(\sigma_2)>0\Big).$$
Clearly,  \eqref{Ehyper}- \eqref{IEChyper} when $p=0$ reduce to \eqref{Hyper}-\eqref{Ichyper}, respectively.

Choi et al. \cite{Choi2014} have introduced and investigated the following extended beta function
\begin{eqnarray}\label{EEbeta}
B(\sigma_1,\sigma_2;p,q)=B_{p,q}(\sigma_1,\sigma_2)=\int_{0}^{1}\,t^{\sigma_1-1}(1-t)^{\sigma_2-1}
 \, \mathtt{e}^{-\frac{p}{t}-\frac{q}{1-t}}\,dt
\end{eqnarray}
\begin{equation*}
  \left(\min\{\Re(p),\,\Re(q) \}>0,\,   \min\{\Re(\sigma_1),\,\Re(\sigma_2) \}>0    \right).
\end{equation*}
 Obviously, $B(\sigma_1,\sigma_2;p,p)=B(\sigma_1,\sigma_2;p)$ in \eqref{Ebeta} and
  $B(\sigma_1,\sigma_2;0,0)=B(\sigma_1,\sigma_2)$ in \eqref{beta}.
They \cite{Choi2014} have introduced
   the following  extended $(p,q)$-hypergeometric function and  extended $(p,q)$-confluent hypergeometric function defined, respectively, by
 \begin{eqnarray}\label{FEhyper}
F_{p,q}(\sigma_1,\sigma_2;\sigma_3;z)=\sum\limits_{n=0}^{\infty}\frac{B_{p,q}(\sigma_2+n, \sigma_3-\sigma_2)}{B(\sigma_2, \sigma_3-\sigma_2)}(\sigma_1)_n\frac{z^n}{n!}
\end{eqnarray}
\begin{equation*}
  \left(p,\, q \in \mathbb{R}_0^+,\,  \Re(\sigma_3)>\Re(\sigma_2)>0   \right)
\end{equation*}
and
\begin{eqnarray}\label{FECH}
\Phi_{p,q}(\sigma_2;\sigma_3;z)=\sum\limits_{n=0}^{\infty}\frac{B_{p,q}(\sigma_2+n, \sigma_3-\sigma_2)}{B(\sigma_2, \sigma_3-\sigma_2)}\frac{z^n}{n!}
\end{eqnarray}
\begin{equation*}
  \left(p,\, q \in \mathbb{R}_0^+,\,  \Re(\sigma_3)>\Re(\sigma_2)>0   \right).
\end{equation*}
They \cite{Choi2014} presented the following integral representations
\begin{equation}\label{FEhyperInt}
\aligned
& F_{p,q}(\sigma_1,\sigma_2;\sigma_3;z)\\
&=\frac{1}{B(\sigma_2, \sigma_3-\sigma_2)}\int_0^1t^{\sigma_2-1}(1-t)^{\sigma_3-\sigma_2-1}(1-zt)^{-\sigma_1}\exp\Big(-\frac{p}{t}-\frac{q}{1-t}\Big)\,dt
\endaligned
\end{equation}
\begin{equation*}
  \left(p,\, q \in \mathbb{R}_0^+,\,  \Re(\sigma_3)>\Re(\sigma_2)>0,\, |\arg(1-z)|<\pi  \right)
\end{equation*}
and
\begin{equation}\label{FEChyper}
\aligned
& \Phi_{p,q}(\sigma_2;\sigma_3;z)\\
&=\frac{1}{B(\sigma_2, \sigma_3-\sigma_2)}\int_0^1t^{\sigma_2-1}(1-t)^{\sigma_3-\sigma_2-1}\exp\Big(zt-\frac{p}{t}-\frac{q}{1-t}\Big)\,dt
\endaligned
\end{equation}
\begin{equation*}
  \left(p,\, q \in \mathbb{R}_0^+,\,  \Re(\sigma_3)>\Re(\sigma_2)>0 \right).
\end{equation*}

Parmar et al. \cite[Eq. (1.13)]{Parmar} introduced the following extended beta function
\begin{eqnarray}\label{pbeta}
B_v(\sigma_1,\sigma_2;p)=\sqrt{\frac{2p}{\pi}}\int_0^1\, t^{\sigma_1-\frac{3}{2}}(1-t)^{\sigma_2-\frac{3}{2}}
K_{v+\frac{1}{2}}\Big(\frac{p}{t(1-t)}\Big)dt \quad (\Re(p)>0),
\end{eqnarray}
where $K_{v}(\cdot)$ is the modified Bessel function of order $v$. By recalling the following identity (see, e.g.,
  \cite[Entry 10.39.2]{Olver})
\begin{equation}\label{K-bessel}
 K_{\frac{1}{2}}(z)=\sqrt{\frac{\pi}{2z}}\, \mathtt{e}^{-z},
\end{equation}
it is obvious that $B_0(\sigma_1,\sigma_2;p)=B(\sigma_1,\sigma_2;p)$ in \eqref{Ebeta}.
 They \cite{Parmar} defined the following extended hypergeometric function $F_{p,v}$ and  extended confluent hypergeometric function
 $\Phi_{p,v}$
 \begin{eqnarray}\label{phyper}
F_{p,v}\Big(\sigma_1,\sigma_2;\sigma_3;z\Big)=\sum_{n=0}^\infty(\sigma_1)_n\frac{B_v(\sigma_2+n,\sigma_3-\sigma_2;p)}{B(\sigma_2,\sigma_3-\sigma_2)}
\frac{z^n}{n!}
\end{eqnarray}
 $$\Big(p,\, v  \in \mathbb{R}_0^+,\, \Re(\sigma_3>\Re(\sigma_2)>0,\,|z|<1\Big)$$
 and
\begin{eqnarray}\label{pchyper}
\Phi_{p,v}\Big(\sigma_2;\sigma_3;z\Big)=\sum_{n=0}^\infty\frac{B_v(\sigma_2+n,\sigma_3-\sigma_2;p)}{B(\sigma_2,\sigma_3-\sigma_2)}
\frac{z^n}{n!}
\end{eqnarray}
 $$\Big(p,\, v \in \mathbb{R}_0^+,\,    \Re(\sigma_3)>\Re(\sigma_2)>0\Big)$$
and presented their integral representations
\begin{equation}\label{pihyper}
\aligned
& F_{p,v}\Big(\sigma_1,\sigma_2;\sigma_3;z\Big)=\sqrt{\frac{2p}{\pi}}\frac{1}{B(\sigma_2,\sigma_3-\sigma_2)}\\
&\hskip 10mm  \times \int_0^1t^{\sigma_2-\frac{3}{2}}(1-t)^{\sigma_3-\sigma_2-\frac{3}{2}}
(1-zt)^{-\sigma_1}K_{v+\frac{1}{2}}\Big(\frac{p}{t(1-t)}\Big)\,dt
\endaligned
\end{equation}
 $$\Big(p,\, v \in \mathbb{R}_0^+,\, \Re(\sigma_3)>\Re(\sigma_2)>0, \, |\arg(1-z)|<\pi\Big)$$
and
\begin{equation}\label{pichyper}
\aligned
&\Phi_{p,v}\Big(\sigma_2;\sigma_3;z\Big)=\sqrt{\frac{2p}{\pi}}\frac{1}{B(\sigma_2,\sigma_3-\sigma_2)}\\
& \hskip 10mm \times \int_0^1t^{\sigma_2-\frac{3}{2}}(1-t)^{\sigma_3-\sigma_2-\frac{3}{2}}
\exp(zt)K_{v+\frac{1}{2}}\Big(\frac{p}{t(1-t)}\Big)\,dt
\endaligned
\end{equation}
 $$\Big(p,\, v\in \mathbb{R}_0^+,\,\Re(\sigma_3)>\Re(\sigma_2)>0\Big).$$
They also obtained the following transformation formula for the extended confluent hypergeometric function
\begin{eqnarray}\label{trans}
\Phi_{p,v}(\sigma_2,\sigma_3;\gamma;z)= \mathtt{e}^z\, \Phi_{p,v}\Big(\sigma_3-\sigma_2;\sigma_3;-z\Big).
\end{eqnarray}
Obviously, due to   \eqref{K-bessel},  equations \eqref{phyper}-\eqref{pichyper}
reduce, respectively,  to \eqref{Ehyper}-\eqref{IEChyper}.

\vskip 3mm

Whittaker \cite{Whittaker1}    introduced the so-called
Whittaker function
\begin{eqnarray}\label{Whittaker}
M_{\lambda,\rho}(z)=z^{\rho+\frac{1}{2}}\,\exp\Big(\hskip -2mm -\frac{z}{2}\Big)\Phi\Big(\rho-\lambda+\frac{1}{2};2\rho+1;z\Big)
\end{eqnarray}
\begin{equation*}
  \left(\Re(\rho)>-\frac{1}{2},\, \Re(\rho\pm\lambda)>-\frac{1}{2} ;\, z \in \mathbb{C}\setminus (-\infty,0]  \right),
\end{equation*}
where  $\Phi$ is the confluent hypergeometric function in \eqref{Chyper} and  which is a modified solution of the Whittaker's equation so that formulas involving the solutions can be
more symmetric (see, e.g., \cite[Chapter XVI]{Whittaker}; see also \cite[p. 39]{Sr-Ma}).

\vskip 3mm

Nagar et al. \cite{Nagar} defined the following extended Whittaker function
\begin{eqnarray}\label{EWhittaker}
M_{p,\lambda,\rho}(z)=z^{\rho+\frac{1}{2}}\exp\Big(\hskip -2mm-\frac{z}{2}\Big)\Phi_p\Big(\rho-\lambda+\frac{1}{2};2\rho+1;z\Big)
\end{eqnarray}
\begin{equation*}
  \left(p \in \mathbb{R}_0^+;\, \Re(\rho)>-\frac{1}{2};\,  \Re(\rho\pm\lambda)>-\frac{1}{2};\, z \in \mathbb{C}\setminus (-\infty,0] \right).
\end{equation*}
where $\Phi_p$ is the extended confluent hypergeometric function in \eqref{ECH}.

\vskip 3mm

 Rahman et al. \cite{Rahman} have introduced and investigated the following extended $(p,q)$-Whittaker function
\begin{eqnarray}\label{FWhittaker}
M_{p,q,\lambda,\rho}(z)=z^{\rho+\frac{1}{2}}\exp\Big( \hskip -2mm -\frac{z}{2}\Big)\Phi_{p,q}\Big(\rho-\lambda+\frac{1}{2};2\rho+1;z\Big)
\end{eqnarray}
\begin{equation*}
  \left(p,\,q \in \mathbb{R}_0^+;\, \Re(\rho)>-\frac{1}{2};\,  \Re(\rho\pm\lambda)>-\frac{1}{2};\, z \in \mathbb{C}\setminus (-\infty,0] \right),
\end{equation*}
where  $\Phi_{p,q}$ is the extended $(p,q)$-confluent hypergeometric function in \eqref{FECH}.

\vskip 3mm

Here we introduce the following extended Whittaker function
\begin{eqnarray}\label{PWhittaker}
M_{p,v,\lambda,\rho}(z)=z^{\rho+\frac{1}{2}}\exp\Big(\hskip -2mm -\frac{z}{2}\Big)\Phi_{p,v}\Big(\rho-\lambda+\frac{1}{2};2\rho+1;z\Big)
\end{eqnarray}
\begin{equation*}
  \left(p,\,v \in \mathbb{R}_0^+;\, \Re(\rho)>-\frac{1}{2};\,  \Re(\rho\pm\lambda)>-\frac{1}{2};\, z \in \mathbb{C}\setminus (-\infty,0] \right),
\end{equation*}
where $\Phi_{p,v}$ is the extended confluent hypergeometric function  in \eqref{pchyper}. Then we investigate
certain formulas  involving the extended Whittaker function \eqref{PWhittaker} such as integral representations,
a transformation formula, Mellin transform, and a differential formula. Some special cases of our results are also considered.

\vskip 3mm

It is remarked in passing that $M_{p,0,\lambda,\rho}(z)=M_{p,\lambda,\rho}(z)$ in \eqref{EWhittaker} and
   $M_{0,0,\lambda,\rho}(z)=M_{\lambda,\rho}(z)$ in \eqref{Whittaker};
From \eqref{trans},  the extended $(p,v)$-Whittaker function \eqref{PWhittaker}  can also be expressed in the following form
\begin{eqnarray}\label{rmk-1}
M_{p,v,\lambda,\rho}(z)=z^{\rho+\frac{1}{2}}\,\exp\left(\frac{z}{2}\right)\,\Phi_{p,v}\Big(\rho+\lambda+\frac{1}{2};2\rho+1;-z\Big).
\end{eqnarray}

\section{Formulas involving the extended Whittaker function \eqref{PWhittaker}  }

Here we establish certain formulas  involving the extended Whittaker function \eqref{PWhittaker} such as integral representations,
a transformation formula, Mellin transform, and a differential formula. Some special cases of our results are also considered.

\vskip 3mm
\begin{theorem}\label{Th1}
Let $p,\,v \in \mathbb{R}_0^+$, $\Re(\rho)>\Re(\rho\pm\lambda)>-\frac{1}{2}$, and $z \in \mathbb{C}\setminus (-\infty,0]$.
Also let $a,\,b \in \mathbb{R}$ with $b>a$.
Then each of the following integral representations holds.
\begin{equation}\label{int1}
\aligned
M_{p,v,\lambda,\rho}(z)
= & \frac{z^{\rho+\frac{1}{2}}\exp(-\frac{z}{2})\sqrt{2p}}{\sqrt{\pi}\,B(\rho-\lambda+\frac{1}{2},\rho+\lambda+\frac{1}{2})}\\
&\times\int_0^1 t^{\rho-\lambda-1}(1-t)^{\rho+\lambda-1}\exp(zt)K_{v+\frac{1}{2}}\Big(\frac{p}{t(1-t)}\Big)\,dt;
\endaligned
\end{equation}

\begin{equation}\label{int2}
\aligned
M_{p,v,\lambda,\rho}(z)
   &=\frac{z^{\rho+\frac{1}{2}}\exp(\frac{z}{2})\sqrt{2p}}{\sqrt{\pi}\, B(\rho-\lambda+\frac{1}{2},\rho+\lambda+\frac{1}{2})}\\
&\times\int_0^1 u^{\rho+\lambda-1}(1-u)^{\rho-\lambda-1}\exp(-zu)K_{v+\frac{1}{2}}\Big(\frac{p}{u(1-u)}\Big)\,du;
\endaligned
\end{equation}

\begin{equation}\label{int3}
\aligned
M_{p,v,\lambda,\rho}(z)= & \frac{(b-a)^{-2\rho-1}z^{\rho+\frac{1}{2}}
\exp(-\frac{z}{2})\sqrt{2p}}{\sqrt{\pi}\,B(\rho-\lambda+\frac{1}{2},\rho+\lambda+\frac{1}{2})}\int_a^b (u-a)^{\rho-\lambda-1}(b-u)^{\rho+\lambda-1}\\
&\times\exp\Big(\frac{z(u-a)}{b-a}\Big)K_{v+\frac{1}{2}}\Big(\frac{p(b-a)^2}{(u-a)(b-u)}\Big)\,du;
\endaligned
\end{equation}

\begin{equation}\label{int4}
\aligned
M_{p,v,\lambda,\rho}(z) = & \frac{z^{\rho+\frac{1}{2}}
\exp(-\frac{z}{2})\sqrt{2p}}{\sqrt{\pi}B(\rho-\lambda+\frac{1}{2},\rho+\lambda+\frac{1}{2})}\int_0^\infty u^{\rho-\lambda-1}(1+u)^{-2\rho}\\
&\times\exp\Big(\frac{zu}{1+u}\Big)K_{v+\frac{1}{2}}\Big(\frac{p(1+u)^2}{u}\Big)du;
\endaligned
\end{equation}

\begin{equation}\label{int5}
\aligned
M_{p,v,\lambda,\rho}(z) = & \frac{2^{-2\rho-1}z^{\rho+\frac{1}{2}}
\sqrt{2p}}{\sqrt{\pi} B(\rho-\lambda+\frac{1}{2},\rho+\lambda+\frac{1}{2})}\int_{-1}^{1} (1+u)^{\rho-\lambda-1}(1-u)^{\rho+\lambda-1}\\
&\times\exp\Big(\frac{zu}{2}\Big)K_{v+\frac{1}{2}}\Big(\frac{2p}{(1+u)(1-u)}\Big)\,du.
\endaligned
\end{equation}
\end{theorem}

\begin{proof}
 By using the integral representation \eqref{pichyper} in the definition \eqref{PWhittaker},
 we obtain \eqref{int1}. Now, by setting $t=1-u$, $t=\frac{u-a}{b-a}$, and $t=\frac{u}{1+u}$ in \eqref{int1}, we get \eqref{int2},
 \eqref{int3}, and \eqref{int4}, respectively.
Setting $a=-1$ and $b=1$ in \eqref{int3} yields \eqref{int5}.
\end{proof}

\vskip 3mm
\begin{theorem}
The following transformation formula for the extended $(p,v)$-Whittaker function \eqref{PWhittaker} holds.
\begin{eqnarray}
M_{p,v,\lambda,\rho}(-z)=(-1)^{\rho+\frac{1}{2}}M_{p,v,-\lambda,\rho}(z)
\end{eqnarray}
\begin{equation*}
  \left(p,\,v \in \mathbb{R}_0^+;\, \Re(\rho)>-\frac{1}{2};\,  \Re(\rho\pm\lambda)>-\frac{1}{2};\, z \in \mathbb{C}\setminus \mathbb{R} \right).
\end{equation*}
\end{theorem}

\begin{proof}
Replacing $z$ by $-z$ in (\ref{PWhittaker}) and using \eqref{rmk-1} , we get
 the desired result.
\end{proof}

\vskip 3mm

\begin{theorem}\label{Th2}
The following Mellin transformation holds.
\begin{equation}
\aligned
& \mathfrak{M}\{M_{p,v,\lambda,\rho}(z);p\rightarrow r \}\\
&=\frac{z^{\rho+\frac{1}{2}-r}\exp(-\frac{z}{2})2^{r-1}\Gamma(\frac{r-v}{2})\Gamma(\frac{r+v+1}{2})
B(\rho+r-\lambda-\frac{1}{2},
\rho+r+\lambda-\frac{1}{2})}
{\sqrt{\pi} \, B(\rho-\lambda+\frac{1}{2},\rho+\lambda+\frac{1}{2})}\\
  &\hskip 3mm  \times \Phi\Big(\rho+r-\lambda-\frac{1}{2};2\rho+2r;z\Big)
\endaligned
\end{equation}
\begin{equation*}
  \left(\Re(r-v)>0,\, \Re(r+v)>-1 ,\, \Re(\rho+r\pm\lambda)> \frac{1}{2},\,   z \in \mathbb{C}\setminus (-\infty,0]  \right).
\end{equation*}

\end{theorem}

\begin{proof}
Using the integral representation in (\ref{int1}) and changing the order of integrations, we get
\begin{equation}\label{mellin}
\aligned
\mathfrak{M}\{M_{p,v,\lambda,\rho}(z);p\rightarrow r\}:=&\int_0^\infty p^{r-1}M_{p,v,\lambda,\rho}(z)\,dp\\
=& \frac{z^{\rho+\frac{1}{2}}\exp(-\frac{z}{2})\sqrt{2}}{\sqrt{\pi}\beta(\rho-\lambda+\frac{1}{2},\rho+\lambda+\frac{1}{2})}
\int_0^1t^{\rho-\lambda-1}
(1-t)^{\rho+\lambda-1}\, \mathtt{e}^{zt}\\
&\times\left\{\int_0^\infty p^{r-\frac{1}{2}}K_{v+\frac{1}{2}}\Big(\frac{p}{t(1-t)}\Big)dp\right\}\,dt.
\endaligned
\end{equation}

Using a known integral formula involving $K_{v}$ (see, e.g., \cite[Entry 10.43.19]{Olver}; see also \cite{Parmar}), we have
\begin{equation}\label{mellin1}
\aligned
& \int_0^\infty p^{r-\frac{1}{2}}K_{v+\frac{1}{2}}\Big(\frac{p}{t(1-t)}\Big)dp= t^{r-\frac{1}{2}}(1-t)^{r-\frac{1}{2}}\int_0^\infty u^{r-\frac{1}{2}}K_{v+\frac{1}{2}}\Big(u\Big)\,du\\
&\hskip 20mm =t^{r-\frac{1}{2}}(1-t)^{r-\frac{1}{2}}2^{r-\frac{3}{2}}\Gamma\left(\frac{r-v}{2}\right)\Gamma\left(\frac{r+v+1}{2}\right)
\endaligned
\end{equation}
\begin{equation*}
  \left(\Re(r-v)>0,\, \Re(r+v)>-1  \right).
\end{equation*}

 Using \eqref{mellin1} in \eqref{mellin}, we obtain
\begin{equation}\label{Mellin-exp}
\aligned
& \mathfrak{M}\{M_{p,v,\lambda,\rho}(z);p\rightarrow r\}\\
&\hskip 5mm = \frac{z^{\rho+\frac{1}{2}-r}\exp(-\frac{z}{2})2^{r-1}\Gamma(\frac{r-v}{2})\Gamma(\frac{r+v+1}{2})}
    {\sqrt{\pi}B(\rho-\lambda+\frac{1}{2},\rho+\lambda+\frac{1}{2})}
\int_0^1t^{\rho+r-\lambda-\frac{3}{2}}
(1-t)^{\rho+r+\lambda-\frac{3}{2}}\,\mathtt{e}^{zt}\,dt.
\endaligned
\end{equation}
Using \eqref{Ichyper}, we find
\begin{equation}\label{beta-exp}
\aligned
 & \int_0^1t^{\rho+r-\lambda-\frac{3}{2}}
(1-t)^{\rho+r+\lambda-\frac{3}{2}}\,\mathtt{e}^{zt}\,dt \\
 & \hskip 5mm = \frac{\Gamma(\rho+r-\lambda-\frac{1}{2})\,  \Gamma(\rho+r+\lambda-\frac{1}{2}) }{\Gamma(2\rho+2r-1)}\,
   \Phi (\rho+r-\lambda-\frac{1}{2}; 2\rho+2r-1;z)
   \endaligned
\end{equation}
\begin{equation*}
\left(  \Re(\rho+r \pm \lambda)>\frac{1}{2}\right).
\end{equation*}
Applying \eqref{beta-exp} to \eqref{Mellin-exp},
 we obtain the desired result.

\end{proof}

\vskip 3mm

\begin{theorem}\label{Th3} Let $p,\,v  \in \mathbb{R}^+_0$, $2\alpha>\mu$, $\Re(\delta+\rho)>-\frac{1}{2}$, and  $\Re(\rho\pm\lambda)>-\frac{1}{2}$. Then
\begin{equation}\label{Th3-eq1}
\aligned
& \int_0^\infty \, x^{\delta-1}\, \mathtt{e}^{-\alpha x}\, M_{p,v,\lambda,\rho}(\mu x)\,dx = \Gamma \left(\delta+\rho+ \frac{1}{2}\right) \\
 &\hskip 5mm \times \mu^{\rho + \frac{1}{2}}\, \left(\alpha + \frac{\mu}{2}\right)^{-\delta-\rho-\frac{1}{2}}\,
F_{p,v}\Big(\delta+\rho+ \frac{1}{2},\,\rho-\lambda+\frac{1}{2}; 2\rho+1;\frac{2\mu}{2\alpha+\mu}\Big).
\endaligned
\end{equation}

\end{theorem}

\begin{proof}
Let $\mathcal{L}$ be the left side of \eqref{Th3-eq1}.
Using the integral representation \eqref{int1} and changing the order of integrations,
which can be verified under the conditions here, we obtain
\begin{equation}\label{Th3-pf1}
  \aligned
 & \mathcal{L} = \sqrt{\frac{2p}{\pi}}\,\frac{\mu^{\rho + \frac{1}{2}}}{B\left(\rho-\lambda+\frac{1}{2},\rho+\lambda+\frac{1}{2}\right)}
 \,  \int_0^1t^{\rho-\lambda-1}(1-t)^{\rho+\lambda-1}
  K_{v+\frac{1}{2}}\Big(\frac{p}{t(1-t)}\Big)\\
 &\hskip 15mm  \times\, \left[\int_{0}^{\infty}\, x^{\delta +\rho -\frac{1}{2}}\,
    \exp \left\{- \left(\alpha + \frac{\mu}{2}-\mu t\right)x\right\}\,dx  \right]\, dt.
  \endaligned
\end{equation}

Using the Euler's gamma function (see, e.g., \cite[Section 1.1]{Sr-Ch-12}), we get
\begin{equation}\label{Th3-pf2}
  \int_{0}^{\infty}\,u^{\alpha-1}\, \exp (-\beta\,u)\,du = \beta^{-\alpha}\, \Gamma(\alpha)
   \quad  \left(\Re(\alpha)>0,\, \beta \in \mathbb{R}^+  \right).
\end{equation}
Applying  \eqref{Th3-pf2}, we have
\begin{equation}\label{Th3-pf3}
\aligned
 & \int_{0}^{\infty}\, x^{\delta +\rho -\frac{1}{2}}\,
    \exp \left\{- \left(\alpha + \frac{\mu}{2}-\mu t\right)x\right\}\,dx \\
 &\hskip 5mm  = \left(\alpha + \frac{\mu}{2}-\mu t\right)^{-\left(\delta + \rho + \frac{1}{2}\right)}\,
       \Gamma \left( \delta + \rho + \frac{1}{2} \right)
 \endaligned
\end{equation}
\begin{equation*}
  \left(2\alpha>\mu,\, \Re(\delta +\rho)> - \frac{1}{2}    \right).
\end{equation*}
Substituting the integral formula \eqref{Th3-pf3} for the inner integral \eqref{Th3-pf1} and using \eqref{pihyper},
we obtain the desired result.

\end{proof}

\vskip 3mm

\begin{theorem}\label{Th5} Let $p,\,v \in \mathbb{R}_0^+$, $\Re(\rho)>-\frac{1}{2}$,   $\Re(\rho\pm\lambda)>-\frac{1}{2}$,
 and $z \in \mathbb{C}\setminus (-\infty,0]$. Also let $n \in \mathbb{N}_0$. Then
\begin{eqnarray}\label{deriv}
\frac{d^n}{dz^n}\Big\{\mathtt{e}^{\frac{z}{2}}z^{-\rho-\frac{1}{2}}M_{p,v,\lambda,\rho}(z)\Big\} =\frac{(\rho-\lambda+\frac{1}{2})_n}{(2\rho+1)_n}\, \mathtt{e}^{\frac{z}{2}}z^{-\rho-\frac{n}{2}-\frac{1}{2}}M_{p,v,\lambda-\frac{n}{2},\rho+\frac{n}{2}}(z).
\end{eqnarray}
\end{theorem}

\begin{proof}
Applying the following known formula (see \cite{Parmar})
\begin{eqnarray}\label{deriv2}
\frac{d^n}{dz^n}\Big\{\Phi_{p,v}(\sigma_2;\sigma_3;z)\Big\}=\frac{(\sigma_2)_n}{(\sigma_3)_n}\Phi_{p,v}
(\sigma_2+n;\sigma_3+n;z) \quad \left( n \in \mathbb{N}_0\right)
\end{eqnarray}
to \eqref{PWhittaker}, we obtain the desired result.
\end{proof}

\section{Special cases and remarks}\label{sec3}
  The results presented here, being very general, can be specialized to yield a number of relatively simple
  identities. We demonstrate only two examples in the following corollaries.

  \vskip 3mm
  Setting $v=0$  in Theorem \ref{Th2}, in view of  \eqref{K-bessel} and \eqref{Whittaker},  we obtain
  \vskip 3mm
  \begin{corollary}\label{cor1}
    Let $\Re(r)>0$, $\Re(\rho+r\pm\lambda)>\frac{1}{2}$,  and   $ z \in \mathbb{C}\setminus (-\infty,0]$. Then
   \begin{eqnarray}\label{cor1-eq1}
\mathfrak{M}\{M_{p,\lambda,\rho}(z);p\rightarrow r\}&=&\frac{z^{-r}\Gamma(r)\, B(\rho+r-\lambda+\frac{1}{2},
\rho+r+\lambda+\frac{1}{2})}
{B(\rho-\lambda+\frac{1}{2},\rho+\lambda+\frac{1}{2})}M_{\lambda,\rho+r}(z).
\end{eqnarray}
\end{corollary}

\vskip 3mm

Setting $v=0$ and then $p=0$ in the result in Theorem \ref{Th3}, we get

\vskip 3mm

\begin{corollary}\label{cor2}
 Let  $2\alpha>\mu$, $\Re(\delta+\rho)>-\frac{1}{2}$, and  $\Re(\rho\pm\lambda)>-\frac{1}{2}$. Then
\begin{equation}\label{cor2-eq1}
\aligned
& \int_0^\infty \, x^{\delta-1}\, \mathtt{e}^{-\alpha x}\, M_{\lambda,\rho}(\mu x)\,dx = \Gamma \left(\delta+\rho+ \frac{1}{2}\right) \\
 &\hskip 5mm \times \mu^{\rho + \frac{1}{2}}\, \left(\alpha + \frac{\mu}{2}\right)^{-\delta-\rho-\frac{1}{2}}\,
{}_2F_{1}\Big(\delta+\rho+ \frac{1}{2},\,\rho-\lambda+\frac{1}{2}; 2\rho+1;\frac{2\mu}{2\alpha+\mu}\Big).
\endaligned
\end{equation}
\end{corollary}

\vskip 3mm
The main results presented here when $(v=0)$ and $(v=0$ and then $p=0)$
are reduced to yield the corresponding results in \cite{Nagar} and the identities for the Whittaker function (see \cite{Whittaker}),
respectively.

\end{document}